\documentclass[12pt,a4paper]{article}
\usepackage{indentfirst}
\usepackage{amsfonts}
\usepackage{amssymb}
\usepackage{mathrsfs}
\usepackage{amsmath}
\usepackage{amsthm}
\usepackage{enumerate}
\usepackage{cite}
\usepackage{cases}
\allowdisplaybreaks
\newtheorem{defn}{Definition}[section]
\newtheorem{thm}[defn]{Theorem}
\newtheorem{lem}[defn]{Lemma}
\newtheorem{prop}[defn]{Proposition}
\newtheorem{cor}[defn]{Corollary}

\newtheorem{re}[defn]{Remark}
\newtheorem{no}[defn]{Notation}
\begin{document}
\title{{\bf Algebras of quotients and Martindale-like quotients of Leibniz algebras}}
\author{\normalsize \bf Chenrui Yao$^{1}$, Yao Ma$^{1}$, Liming Tang$^{2}$,  Liangyun Chen$^{1}$}
 \date{\small $^{1}$ School of Mathematics and Statistics, Northeast Normal University,  \\Changchun 130024,  China\\
$^{2}$ School of Mathematical Sciences,
Harbin Normal University,\\ Harbin 150025, China} \maketitle
\date{}

   {\bf\begin{center}{Abstract}\end{center}}

In this paper, the definitions of algebras of quotients and Martandale-like qoutients of Leibniz algebras are introduced and the interactions between the two quotients are determined. Firstly, some important properties which not only hold for a Leibniz algebras but also can been lifted to its algebras of quotients are investigated. Secondly, for any semiprime Leibniz algebra, its maximal algebra of quotients is constucted and a Passman-like characterization of the maximal algebra is described. Thirdly, the relationship between a Leibniz algebra and the associative algebra which is generated by left and right multiplication operators of the corresponding Leibniz algebras of quotients are examined. Finally, the definition of dense extensions and some vital properties about Leibnia algebras via dense extensions are introduced.

\renewcommand{\thefootnote}{\fnsymbol{footnote}}
\footnote[0]{ Corresponding author(L. Chen): chenly640@nenu.edu.cn.}
\footnote[0]{Supported by  NNSF of China (Nos. 11771069 and 11801066).}

\section{Introduction}
The notion of Leibniz algebras was first introduced by Loday in \cite{L1}, which is a generalization of Lie algebras where the skew-symmetry is omitted. This gives rise to two types of Leibniz algebras: left Leibniz algebra and right Leibniz algebra. A right(resp. left) Leibniz algebra is a vector space $L$ over $\mathbb{F}$ with a bilinear map $[\cdot, \cdot] : L \times L \rightarrow L$ satisfying for all $x, y, z$ in $L$
\[[x, [y, z]] = [[x, y], z] - [[x, z], y].(\rm{resp}.\;\it{[x, [y, z]] = [[x, y], z] + [y, [x, z]]}.)\]
That is to say, the right(resp. left) multiplication operators are derivations of $L$. If $L$ is both a right and left Leibniz algebra, then $L$ is called a symmetric Leibniz algebra, introduced by Geoffrey and Gaywalee in \cite{GG}. In recent years, study about Leibniz algebras is becoming more and more popular and many results of Lie algebras have been generalized to Leibniz algebras, such as references \cite{B1, BCG}. For more references about Leibniz algebras, one can see \cite{BH, BF, BMS}.

The notion of ring of quotients was introduced by Utumi in \cite{U1}, which played an important role in the development of the theories of associative and commutative rings. He showed that every associative ring without right zero divisors has a maximal left quotients ring and constructed it. Inspired by \cite{U1}, Siles Molina studied the algebras of quotients of Lie algebras in \cite{S1}. Martindale rings of quotients were introduced by Martindale in 1969 for prime rings in \cite{M1}, which was designed for applications to rings satisfying a generalized polynomial identity. In \cite{GG1}, Garc\'{i}a and G\'{o}mez defined Martindale-like quotients for Lie triple systems with respect to power filters of sturdy ideals and constructed the maximal system of quotients in the nondegenerate cases. More research about quotients of Lie systems refer in references \cite{GC, MCL}. In \cite{M2}, Mart\'{i}nez derived a necessary and sufficient Ore type condition for a Jordan algebra to have a ring of fractions, which is the origin of algebras of quotients of Jordan systems. In \cite{AGG, AM, GG2, M3}, authors researched algebras of quotients and Martindale-like quotients of Jordan systems respectively. Inspired by \cite{S1, U1}, we will introduce the notion of algebras of quotients of Leibniz algebras and prove the properties, such as semiprimeness and primeness,  can be lifted from a Leibniz algebra to its algebra of quotients in this paper.

As we know, one popular topic of research is the relationship between algebras of quotients of non-associative algebras and associative algebras in recent years. In \cite{PS}, authors examined the relationship between associative and Lie algebras of quotients. Inspired by them, we'll explore whether there exists a relationship between associative and Leibniz algebras of quotients. The answer is affirmative and we will describe in detail in the following.

The paper is organised as follows: in Section \ref{se:2}, some basic definition are introduced and some elemental propositions are given. In Section \ref{se:3}, the definition of quotients and ideally absorbed are introduced first, which proved to be equivalent in the following, presented as Theorem \ref{thm:3.10}. What's more, the concept of Martindale-like quotients are introduced and the relationship between them are determined(See Proposition \ref{prop:3.16}). In Section \ref{se:4}, we focus on semiprime Leibniz algebras and construct the maximal algebras of quotients of them and characterize them. In Section \ref{se:5}, we mainly study the relationship between algebras of quotients of Leibniz algebras and associative algebras generated by left and right multiplication operators of Leibniz algebras, and get the main result that $\mathscr{A}(Q)$, the associative algebra generated by left and right multiplication operators of $Q$, is a left quotients of $\mathscr{A}_{0}$ if $Q$ is an algebra of quotients of $L$, where $\mathscr{A}_{0}$ denotes the subalgebra of $\mathscr{A}(Q)$ satisfying $\mathscr{A}_{0}(L) \subseteq L$, written as Theorem \ref{thm:5.11}. In Section \ref{se:6}, we introduce the concept of dense extension first of all. Next, we get a proposition about dense extension, Proposition \ref{prop:6.5}, that every $I \subseteq Q$ is also dense for every essential ideal $I$ of $L$ if $L \subseteq Q$ is a dense extension and $Q$ is an algebra of quotients of $L$. Also, we find the possible converse, Proposition \ref{prop:6.9}, to Theorem \ref{thm:5.11} in the presence of dense extensions of Leibniz algebras.

In Sections \ref{se:2}, \ref{se:3} and \ref{se:4}, we suppose that the Leibniz algebras are right Leibniz algebras. While in Sections \ref{se:5} and \ref{se:6}, we assume that all Leibniz algebras are symmetric Leibniz algebras.
\section{Preliminaries}\label{se:2}
\begin{defn}
Let $L$ be a Leibniz algebra and $H$ a subset of $L$.
\begin{enumerate}[(1)]
\item The vector space $lan_{L}(H) = \{x \in L\;|\;[x, y] = 0,\;\forall y \in H\}$ is called the left annihilator of $H$ in $L$. We denote $lan_{L}(L)$ by $lan(L)$.
\item The vector space $ran_{L}(H) = \{x \in L\;|\;[y, x] = 0,\;\forall y \in H\}$ is called the right annihilator of $H$ in $L$. We denote $ran_{L}(L)$ by $ran(L)$.
\item The vector space $\rm{Ann}\it{_{L}(H)} = lan_{L}(H) \cap ran_{L}(H)$ is called the annihilator of $H$ in $L$. We denote $\rm{Ann}\it{_{L}(L)}$ by $\rm{Ann}\it{(L)}$.
\end{enumerate}
\end{defn}
\begin{re}\label{re:2.4}
\begin{enumerate}[(1)]
\item $ran_{L}(H)$ is an ideal of $L$ if $H$ is a right ideal of $L$.
\item $\rm{Ann}\it{_{L}(H)}$ is an ideal of $L$ if $H$ is an ideal of $L$.
\end{enumerate}
\end{re}
\begin{proof}
(1) For any $z \in L$, $x \in ran_{L}(H)$ and $y \in H$,
\[[y, [x, z]] = [[y, x], z] - [[y, z], x] = 0,\]
\[[y, [z, x]] = [[y, z], x] - [[y, x], z] = 0,\]
which imply that both $[x, z]$ and $[z, x]$ belong to $ran_{L}(H)$, i.e., $ran_{L}(H)$ is an ideal of $L$.

(2) For any $z \in L$, $x \in \rm{Ann}\it{_{L}(H)}$ and $y \in H$,
\[[y, [x, z]] = [[y, x], z] - [[y, z], x] = 0,\]
\[[[x, z], y] = [[x, y], z] + [x, [z, y]] = 0,\]
which imply that $[x, z] \in \rm{Ann}\it{_{L}(H)}$, i.e., $\rm{Ann}\it{_{L}(H)}$ is a right ideal of $L$. Similarly, we can show that $\rm{Ann}\it{_{L}(H)}$ is a left ideal of $L$. Thus, $\rm{Ann}\it{_{L}(H)}$ is an ideal of $L$.
\end{proof}
\begin{defn}
Suppose that $L$ is a Leibniz algebra. Then $L$ is semiprime if for any nonzero ideal $I$ of $L$, $[I, I] \neq \{0\}$ and prime if for any two nonzero ideals $I$, $J$ of $L$, $[I, J] \neq \{0\}$.
\end{defn}
\begin{prop}\label{prop:2.6}
For a Leibniz algebra $L$, we have $\rm{Ann}\it{(L)} = ran(L) = \{\rm{0}\}$ if $L$ is semiprime.
\end{prop}
\begin{proof}
Otherwise, $ran(L)$ is a nonzero ideal of $L$ such that $[ran(L), ran(L)] \subseteq [L, ran(L)] = \{0\}$ according to the definition of $ran(L)$, which contradicts with $L$ being semiprime. Thus $ran(L) = \{0\}$, consequently $\rm{Ann}\it{(L)} = \{\rm{0}\}$.
\end{proof}
\begin{defn}
An ideal $I$ of a Leibniz algebra $L$ is called essential if for any nonzero ideal $J$ of $L$, $I \cap J \neq \{0\}$ and sturdy if $\rm{Ann}\it{_{L}(I)} = \{\rm{0}\}$.
\end{defn}
\begin{prop}\label{prop:2.8}
Suppose that $L$ is a Leibniz algebra and $I$ an ideal of $L$.
\begin{enumerate}[(1)]
\item If $\rm{Ann}\it{_{L}(I)} = \{\rm{0}\}(resp. \it{ran_{L}(I)} = \{\rm{0}\})$, then $I$ is essential.
\item If $L$ is semiprime, then $I \cap \rm{Ann}\it{_{L}(I)} = \{\rm{0}\}(resp. \it{I \cap ran_{L}(I)} = \{\rm{0}\})$ and $I$ is essential if and only if $\rm{Ann}\it{_{L}(I)} = \{\rm{0}\}(resp. \it{ran_{L}(I)} = \{\rm{0}\})$.
\end{enumerate}
\end{prop}
\begin{proof}
(1) Otherwise, there exists a nonzero ideal $J$ of $L$ such that $I \cap J = \{0\}$. Take $0 \neq x \in J$. We get
\[[x, I] \subseteq I \cap J = \{0\},\quad[I, x] \subseteq I \cap J = \{0\},\]
which implies that $0 \neq x \in \rm{Ann}\it{_{L}(I)}(\rm{resp}. 0 \neq \it{x \in ran_{L}(I)})$. Contradiction.

(2) Suppose that $I \cap \rm{Ann}\it{_{L}(I)} \neq \{\rm{0}\}(resp. \it{I \cap ran_{L}(I)} \neq \{\rm{0}\})$. Then we can see that $I \cap \rm{Ann}\it{_{L}(I)}(\rm{resp}. \it{I \cap ran_{L}(I)})$ is a nonzero ideal of $L$ satisfying
\[[I \cap \rm{Ann}\it{_{L}(I)}, I \cap \rm{Ann}\it{_{L}(I)}] \subseteq [I, \rm{Ann}\it{_{L}(I)}] = \{\rm{0}\},\]
\[(\rm{resp}. \it{[I \cap ran_{L}(I), I \cap ran_{L}(I)] \subseteq [I, ran_{L}(I)]} = \{\rm{0}\}),\]
which contradicts with $L$ being semiprime.

Now suppose that $I$ is essential. If $\rm{Ann}\it{_{L}(I)} \neq \{\rm{0}\}(resp. \it{ran_{L}(I)} \neq \{\rm{0}\})$, then $I \cap \rm{Ann}\it{_{L}(I)} \neq \{\rm{0}\}(resp. \it{I \cap ran_{L}(I)} \neq \{\rm{0}\})$, contradiction. Combining with (1), we come to the conclusion.
\end{proof}
\begin{prop}\label{prop:2.9}
Suppose that $L$ is a Leibniz algebra. Then $L$ is prime if and only if for every nonzero ideal $I$ of $L$, $ran_{L}(I) = \{0\}$.
\end{prop}
\begin{proof}
Suppose that $L$ is prime. If there exists a nonzero ideal $J$ of $L$ with $ran_{L}(J) \neq \{0\}$, then $[J, ran_{L}(J)] = \{0\}$ according to the definition of $ran_{L}(J)$, which contradicts with $L$ being prime.

Conversely, suppose that for every nonzero ideal $I$ of $L$, $ran_{L}(I) = \{0\}$. If $L$ isn't prime, there exists two nonzero ideals $I, J$ of $L$ such that $[J, I] = \{0\}$, which implies that $I \subseteq ran_{L}(J)$. Contradiction.
\end{proof}
\section{Algebras of quotients of Leibniz algebras}\label{se:3}
Inspired by the notion of algebras of quotients of associative algebras and Lie algebras in \cite{S1, U1}, we introduce the notion of algebras of quotients of Leibniz algebras.

Suppose that $L$ is a subalgebra of the Leibniz algebra $Q$. For any $q \in Q$, set
\[_{L}(q) = \mathbb{F}q + \left\{\sum_{i = 1}^{n}\xi_{i}(q)\mid \xi_{i} \in \mathscr{A}(L)\;\rm{with}\;\it{n} \in \mathbb{N}\right\}\]
where $\mathscr{A}(L)$ denotes the associative algebra generated by $\{R_{x}$, $L_{y} \mid x, y \in L\}$.

That is, $_{L}(q)$ is the linear span in $Q$ of the elements of the form $\xi(q)$ and $q$ where $\xi \in \mathscr{A}(L)$. It's obvious that for any $x, y \in L$, both $R_{x}(_{L}(q))$ and $L_{y}(_{L}(q))$ still be contained in $_{L}(q)$. Now define
\[(L : q) = \{x \in L \mid [x, _{L}(q)] \subseteq L,\;[_{L}(q), x] \subseteq L\}.\]
Clearly, $(L : q)$ is equal to $L$ if $q \in L$.
\begin{prop}\label{prop:3.1}
Suppose that $L$ is a subalgebra of $Q$. Then $(L : q)$ is an ideal of $L$. Moreover, it's maximal among the ideals $I$ of $L$ such that $[I, q] \subseteq L$ and $[q, I] \subseteq L$.
\end{prop}
\begin{proof}
For any $x \in (L : q)$ and $y \in L$, we have
\[[[x, y], _{L}(q)] = [[x, _{L}(q)], y] + [x, [y, _{L}(q)]] \subseteq L,\]
\[[_{L}(q), [x, y]] = [[_{L}(q), x], y] - [[_{L}(q), y], x] \subseteq L,\]
which imply that $[x, y] \in (L : q)$. Similarly, we have $[y, x] \in (L : q)$. Hence, $(L : q)$ is an ideal of $L$.

Suppose that $I$ is an ideal of $L$ such that $[I, q] \subseteq L$ and $[q, I] \subseteq L$. Our goal is to show that $[I, _{L}(q)] \subseteq L$ and $[_{L}(q), I] \subseteq L$.

For any $y \in _{L}(q)$ with $y \neq q$, there exist $\xi_{1}, \cdots, \xi_{n} \in \mathscr{A}(L)$ such that $y = \sum_{i = 1}^{n}\xi_{i}(q)$. Moreover, for each $1 \leq i \leq n$, there exist $\eta_{i}^{1}, \cdots, \eta_{i}^{r_{i}}$ such that $\xi_{i} = \eta_{i}^{r_{i}}\cdots\eta_{i}^{1}$ where $\eta_{i}^{j} = R_{u}$ or $L_{v}$ for some $u, v \in L$ for every $1 \leq j \leq r_{i}$. Indeed, we only need to show that for each $1 \leq i \leq n$, $[I, \xi_{i}(q)] \subseteq L$ and $[\xi_{i}(q), I] \subseteq L$. Let's show that by induction on $r_{i}$.

When $r_{i} = 1$, if $\eta_{i}^{1} = R_{x}$ for some $x \in L$, we have
\[[I, \xi_{i}(q)] = [I, \eta_{i}^{1}(q)] = [I, R_{x}(q)] = [I, [q, x]] = [[I, q], x] - [[I, x], q] \subseteq L,\]
\[[\xi_{i}(q), I] = [\eta_{i}^{1}(q), I] = [R_{x}(q), I] = [[q, x], I] = [[q, I], x] + [q, [x, I]] \subseteq L,\]
if $\eta_{i}^{1} = L_{y}$ for some $y \in L$, we have
\[[I, \xi_{i}(q)] = [I, \eta_{i}^{1}(q)] = [I, L_{y}(q)] = [I, [y, q]] = [[I, y], q] - [[I, q], y] \subseteq L,\]
\[[\xi_{i}(q), I] = [\eta_{i}^{1}(q), I] = [L_{y}(q), I] = [[y, q], I] = [[y, I], q] + [y, [q, I]] \subseteq L.\]

Suppose that the conclusion is valid when $r_{i} = k$. When $r_{i} = k + 1$, denote $\eta_{i}^{k}\cdots\eta_{i}^{1}(q)$ by $z$. By induction hypothesis, $[I, z] \subseteq L$ and $[z, I] \subseteq L$. If $\eta_{i}^{k + 1} = R_{x}$ for some $x \in L$, we have
\[[I, \xi_{i}(q)] = [I, \eta_{i}^{k + 1}\eta_{i}^{k}\cdots\eta_{i}^{1}(q)] = [I, \eta_{i}^{k + 1}(z)] = [I, [z, x]] = [[I, z], x] - [[I, x], z] \subseteq L,\]
\[[\xi_{i}(q), I] = [\eta_{i}^{k + 1}\eta_{i}^{k}\cdots\eta_{i}^{1}(q), I] = [\eta_{i}^{k + 1}(z), I] = [[z, x], I] = [[z, I], x] + [z, [x, I]] \subseteq L.\]
Similarly, we have both $[I, \xi_{i}(q)]$ and $[\xi_{i}(q), I]$ are contained in $L$ if $\eta_{i}^{k + 1} = L_{y}$ for some $y \in L$.

Therefore, $I \subseteq (L : q)$ and the proof is completed.
\end{proof}
\begin{defn}
Suppose that $L$ is a subalgebra of $Q$. Then $Q$ is called an algebra of quotients, if given $p, q \in Q$ with $p \neq 0$, there exists $x \in L$ such that $[x, p] \neq 0$ and $x \in (L : q)$ or there exists $y \in L$ such that $[p, y] \neq 0$ and $y \in (L : q)$.
\end{defn}
\begin{prop}\label{prop:3.3}
Let $L$ be a subalgebra of $Q$.
\begin{enumerate}[(1)]
\item If $\rm{Ann}\it{(L)} = \{\rm{0}\}$, then $L$ is an algebra of quotients of itself.
\item If $Q$ is an algebra of quotients of $L$, then $\rm{Ann}\it{_{Q}(L)} = \rm{Ann}\it{(L)} = \{\rm{0}\}$.
\end{enumerate}
\end{prop}
\begin{proof}
(1) For any $x, y \in L$ with $x \neq 0$, $x \notin \rm{Ann}\it{(L)}$ by assumption. Then there exists $u \in L$ such that $[u, x] \neq 0$ or $v \in L$ such that $[x, v] \neq 0$. Obviously, $u, v \in (L : y)$.

(2) For any $0 \neq q \in Q$, there exists $x \in L$ such that $[q, x] \neq 0$ or $y \in L$ such that $[y, q] \neq 0$, which implies that $q \notin \rm{Ann}\it{_{Q}(L)}$. Hence $\rm{Ann}\it{_{Q}(L)} = \{\rm{0}\}$. Similarly we have $\rm{Ann}\it{(L)} = \{\rm{0}\}$.
\end{proof}
The above proposition says for a Leibniz algebra $L$, that $\rm{Ann}\it{(L)} = \{\rm{0}\}$ is a sufficient and necessary condition such that $L$ has algebras of quotients.

We will prove that some properties of a Leibniz algebra $L$ can be inherited by its algebras of quotients $Q$. Actually, $Q$ just needs a weaker condition.
\begin{defn}
Suppose that $L$ is a subalgebra of $Q$. Then $Q$ is called a weak algebra of quotients of $L$, if given $0 \neq q \in Q$, there exists $x \in L$ such that $0 \neq [q, x] \in L$ or $y \in L$ such that $0 \neq [y, q] \in L$.
\end{defn}
\begin{re}
Every algebra of quotients of a Leibniz algebra is a weak algebra of quotients.
\end{re}
\begin{prop}\label{prop:3.6}
Let $Q$ be a weak algebra of quotients of $L$.
\begin{enumerate}[(1)]
\item If $I$ is a nonzero ideal of $Q$, then $I \cap L$ is a nonzero ideal of $L$.
\item If $L$ is semiprime(prime), so is $Q$.
\end{enumerate}
\end{prop}
\begin{proof}
(1) Take $0 \neq x \in I$. There exists $y \in L$ such that $0 \neq [x, y] \in L$ or $z \in L$ such that $0 \neq [z, x] \in L$. Note that $I$ is an ideal of $Q$, both $[x, y]$ and $[z, x]$ belong to $I$. Therefore, $0 \neq [x, y] \in I \cap L$ or $0 \neq [z, x] \in I \cap L$. It's straightforward to show that $I \cap L$ is an ideal of $L$.

(2) Suppose that $Q$ isn't semiprime. Then there exists a nonzero ideal $I$ of $Q$ such that $[I, I] = \{0\}$. By (1), $I \cap L$ is a nonzero ideal of $L$ such that $[I \cap L, I \cap L] \subseteq [I, I] = \{0\}$, which contradicts with $L$ being simiprime.

Similarly, we can show that $Q$ is prime if $L$ is prime.
\end{proof}
\begin{defn}
Let $L$ be a subalgebra of $Q$. Then $Q$ is called ideally absorbed into $L$, if for each $0 \neq q \in Q$, there exists an ideal $I$ of $L$ with $\rm{Ann}\it{_{L}(I)} = \{\rm{0}\}$ such that $[I, q] \neq \{0\}$ or $[q, I] \neq \{0\}$ and both $[I, q]$ and $[q, I]$ are contained in $L$.
\end{defn}
\begin{prop}\label{prop:3.8}
Suppose that $L$ is a subalgebra of $Q$. Take $q \in Q$.
\begin{enumerate}[(1)]
\item If $Q$ is an algebra of quotients of $L$, then $(L : q)$ is an essential ideal of $L$. Moreover, $\rm{Ann}\it{_{L}((L : q))} = \{\rm{0}\}$.
\item If $Q$ is ideally absorbed into $L$, then $(L : q)$ is an essential ideal of $L$. Moreover, $\rm{Ann}\it{_{L}((L : q))} = \{\rm{0}\}$.
\end{enumerate}
\end{prop}
\begin{proof}
(1) Let $I$ be a nonzero ideal of $L$. Take $0 \neq x \in I$. Apply that $Q$ is an algebra of quotients of $L$ to find $y \in L$ such that $[x, y] \neq 0, y \in (L : q)$ or $z \in L$ such that $[z, x] \neq 0, z \in (L : q)$. Note that $(L : q)$ is an ideal of $L$, we have both $[x, y]$ and $[z, x]$ belong to $(L : q) \cap I$, which implies that $(L : q) \cap I \neq \{0\}$. Therefore, $(L : q)$ is an essential ideal of $L$.

Suppose that $\rm{Ann}\it{_{L}((L : q))} \neq \{\rm{0}\}$. Then $(L : q) \cap \rm{Ann}\it{_{L}((L : q))} \neq \{\rm{0}\}$ since $(L : q)$ is essential. Take $0 \neq u \in (L : q) \cap \rm{Ann}\it{_{L}((L : q))}$. Apply that $Q$ is an algebra of quotients of $L$ to find $w \in L$ such that $[u, w] \neq 0, w \in (L : q)$ or $v \in L$ such that $[v, u] \neq 0, v \in (L : q)$. However, both $[u, w]$ and $[v, u]$ are equal to $0$ since $u \in \rm{Ann}\it{_{L}((L : q))}$. Contradiction.

(2) Since $Q$ is ideally absorbed into $L$, there exists a nonzero ideal $I$ of $L$ with $\rm{Ann}\it{_{L}(I)} = \{\rm{0}\}$ such that $[I, q] \subseteq L$ and $[q, I] \subseteq L$. According to the proof of Proposition \ref{prop:3.1}, $I \subseteq (L : q)$. Hence, $\rm{Ann}\it{_{L}((L : q))} \subseteq \rm{Ann}\it{_{L}(I)} = \{\rm{0}\}$. By Proposition \ref{prop:2.8} (1), we get $(L : q)$ is an essential ideal of $L$.
\end{proof}
\begin{lem}\label{le:3.9}
Let $Q$ be a weak algebra of quotients of $L$. Let $I$ be an ideal of $L$ with $\rm{Ann}\it{_{L}(I)} = \{\rm{0}\}$. Then there is no nonzero element $x$ in $Q$ such that $[x, I] = [I, x] = \{0\}$.
\end{lem}
\begin{proof}
Suppose that there exists $0 \neq x \in Q$ such that $[x, I] = [I, x] = \{0\}$. Apply that $Q$ is a weak algebra of quotients of $L$ to find $y \in L$ such that $0 \neq [x, y] \in L$ or $z \in L$ such that $0 \neq [z, x] \in L$.

If $0 \neq [x, y] \in L$, we have $[I, [x, y]] \neq \{0\}$ or $[[x, y], I] \neq \{0\}$ since $\rm{Ann}\it{_{L}(I)} = \{\rm{0}\}$. Either of the two situations leads to a contradiction. Indeed,
\[[I, [x, y]] = [[I, x], y] - [[I, y], x] = \{0\},\]
\[[[x, y], I] = [[x, I], y] + [x, [y, I]] = \{0\}.\]

The situation is similarly if $0 \neq [z, x] \in L$. Therefore the proof is completed.
\end{proof}
\begin{thm}\label{thm:3.10}
Suppose that $L$ is a subalgebra of $Q$. Then $Q$ is an algebra of quotients of $L$ if and only if $Q$ is ideally absorbed into $L$.
\end{thm}
\begin{proof}
Suppose that $Q$ is an algebra of quotients of $L$. For any $0 \neq q \in Q$, $(L : q)$ is an ideal of $L$ with $\rm{Ann}\it{_{L}((L : q))} = \{\rm{0}\}$ according to Proposition \ref{prop:3.8} (1). By Lemma \ref{le:3.9}, $[q, (L: q)] \neq \{0\}$ or $[(L: q), q] \neq \{0\}$. Also, both $[q, (L : q)]$ and $[(L : q), q]$ contain in $(L : q)$ according to the definition of $(L : q)$. Hence, $Q$ is ideally absorbed into $L$.

Conversely, suppose that $Q$ is ideally absorbed into $L$. Then $Q$ is a weak algebra of quotients of $L$. For any $p, q \in Q$ with $p \neq 0$, we have $\rm{Ann}\it{_{L}((L : q))} = \{\rm{0}\}$. Hence, $[p, (L : q)] \neq \{0\}$ or $[(L : q), p] \neq \{0\}$ according to Lemma \ref{le:3.9}. Therefore, there exists $x \in (L: q)$ such that $[p, x] \neq 0$ or $y \in (L: q)$ such that $[y, p] \neq 0$. Hence, $Q$ is an algebra of quotients of $L$.
\end{proof}
In the proof of Theorem \ref{thm:3.10}, we can conclude that $Q$ is an algebra of quotients of $L$ if and only if $Q$ is a weak algebra of quotients of $L$ satisfying $\rm{Ann}\it{_{L}((L : q))} = \{\rm{0}\}$ for every $q \in Q$.
\begin{prop}\label{prop:3.11}
Suppose that $Q$ is a weak algebra of quotients of $L$. Then for every ideal $I$ of $L$, $\rm{Ann}\it{_{L}(I)} = \{\rm{0}\} \Rightarrow \rm{Ann}\it{_{Q}(I)} = \{\rm{0}\}$ and $ran_{L}(I) = \{0\} \Rightarrow ran_{Q}(I) = \{0\}$.
\end{prop}
\begin{proof}
Suppose that $0 \neq q \in \rm{Ann}\it{_{Q}(I)}$. By hypothesis, there exists $x \in L$ such that $0 \neq [q, x] \in L$ or $y \in L$ such that $0 \neq [y, q] \in L$.

If $0 \neq [q, x] \in L$, we have $[I, [q, x]] \neq \{0\}$ or $[[q, x], I] \neq \{0\}$ since $\rm{Ann}\it{_{L}(I)} = \{\rm{0}\}$. Either of the two situations leads to a contradiction. Indeed,
\[[I, [q, x]] = [[I, q], x] - [[I, x], q] = \{0\},\]
\[[[q, x], I] = [[q, I], x] + [q, [x, I]] = \{0\}.\]

The situation is similarly if $0 \neq [y, q] \in L$, consequently $\rm{Ann}\it{_{Q}(I)} = \{\rm{0}\}$.

Similarly, we have $ran_{Q}(I) = \{0\}$ if $ran_{L}(I) = \{0\}$.
\end{proof}
\begin{no}
Denote by $\mathscr{J}_{e}(L)$ the set of all essential ideals of a Leibniz algebra $L$.
\end{no}
Clearly, given $I, J \in \mathscr{J}_{e}(L)$, we have $I \cap J \in \mathscr{J}_{e}(L)$. If $L$ is semiprime, then $I^{2} \in \mathscr{J}_{e}(L)$.
\begin{prop}\label{prop:3.13}
Let $L$ be a semiprime Leibniz algebra and $Q$ an algebra of quotients of $L$. Then for every essential ideal $I$ of $L$, we have $Q$ is also an algebra of quotients of $I$.
\end{prop}
\begin{proof}
For any $p, q \in Q$ with $p \neq 0$, we have $(L : q) \in \mathscr{J}_{e}(L)$. So that $ I \cap (L : q) \in \mathscr{J}_{e}(L)$ and consequently $(I \cap (L : q))^{2} \in \mathscr{J}_{e}(L)$. Thus $\rm{Ann}\it{_{Q}((I \cap (L : q))^{2})} = \{\rm{0}\}$ by Proposition \ref{prop:3.11}. So there exist $y, z \in I \cap (L : q)$ such that $[[y, z], p] \neq 0$ or $u, v \in I \cap (L : q)$ such that $[p, [u, v]] \neq 0$. Moreover,
\[[[y, z], _{I}(q)] = [[y, _{I}(q)], z] + [y, [z, _{I}(q)]] \subseteq [[y, _{L}(q)], z] + [y, [z, _{L}(q)]] \subseteq [L, I] + [I, L] \subseteq I,\]
\[[_{I}(q), [y, z]] = [[_{I}(q), y], z] - [[_{I}(q), z], y] \subseteq [[_{L}(q), y], z] - [[_{L}(q), z], y] \subseteq [L, I] \subseteq I,\]
which implies that $[y, z] \in (I : q)$. Similarly $[u, v] \in (I : q)$.

Therefore, $Q$ is an algebra of quotients of $I$.
\end{proof}
Next, let's focus on Martindale-like quotients.
\begin{defn}
A filter $\mathscr{F}$ on a Leibniz algebra is a nonempty family of nonzero ideals such that for any $I_{1}, I_{2} \in \mathscr{F}$ there exists $I \in \mathscr{F}$ such that $I \subseteq I_{1} \cap I_{2}$. Moreover, $\mathscr{F}$ is a power filter if for any $I \in \mathscr{F}$ there exists $K \in \mathscr{F}$ such that $K \subseteq [I, I]$.
\end{defn}
\begin{defn}
A Leibniz algebra of Martindale-like quotients $Q$ of a Leibniz algebra $L$ with respect to a power filter of sturdy ideals $\mathscr{F}$ if $L \subseteq Q$ such that for every nonzero element $q \in Q$ there exists an ideal $I_{q} \in \mathscr{F}$ such that $[q, I_{q}] \neq \{0\}$ or $[I_{q}, q] \neq \{0\}$ and both $[q, I_{q}]$ and $[I_{q}, q]$ contain in $L$.
\end{defn}
\begin{prop}\label{prop:3.16}
Let $L$ be a Leibniz algebra. Then a Martindale-like quotients of $L$ is also an algebra of quotients of $L$. Moreover, if $L$ is semiprime, then any algebra of quotients of $L$ is Martindale-like quotients of $L$.
\end{prop}
\begin{proof}
Suppose that $Q$ is a Martindale-like quotients of $L$. Then $Q$ is ideally absorbed into $L$ by definition and so an algebra of quotients of $L$.

Conversely, suppose that $Q$ is an algebra of quotients of $L$ with $L$ semiprime. Then all sturdy ideals of $L$ form a power filter of $L$. Indeed, for any sturdy ideal $I$ of $L$, we have $[I, I]$ is also a sturdy ideal since $L$ is semiprime. So it just need to take $K = [I, I]$. Hence by hypothesis, $Q$ is also a Martindale-like quotients of $L$.
\end{proof}
\section{The maximal algebra of quotients of a semiprime Leibniz algebra}\label{se:4}
In this section, we will construct a maximal algebra of quotients for every semiprime Leibniz algebra $L$. Maximal in the sense that every algebra of quotients of $L$ can be considered inside this maximal algebra of quotients via a monomorphism which restricted in $L$ is the identity.
\begin{defn}
Given an ideal $I$ of a Leibniz algebra $L$, we say that a linear map $\delta : I \rightarrow L$ is a partial derivation if for any $x, y \in I$ it satisfies:
\[\delta([x, y]) = [\delta(x), y] + [x, \delta(y)].\]
\end{defn}
Denote by $\rm{PDer}\it{(I, L)}$ the set of all partial derivations of $I$ in $L$.
\begin{lem}\label{le:4.2}
Let $L$ be a semiprime Leibniz algebra and consider the set
\[\mathscr{D} := \{(\delta, I)\;|\;I \in \mathscr{J}_{e}(L), \delta \in \rm{PDer}\it{(I, L)}\}.\]
Define on $\mathscr{D}$ the following relation:

$(\delta, I) \equiv (\mu, J)$ if and only if there exists $K \in \mathscr{J}_{e}(L)$, $K \subseteq I \cap J$ such that
\[\delta|_{K} = \mu|_{K}.\]
Then $\equiv$ is an equivalence relation.
\end{lem}
\begin{proof}
(1) For any $(\delta, I) \in \mathscr{D}$, $\delta|_{I} = \delta|_{I}$, so $(\delta, I) \equiv (\delta, I)$.

(2) For any $(\delta, I), (\mu, J) \in \mathscr{D}$ and $(\delta, I) \equiv (\mu, J)$, there exists $K \in \mathscr{J}_{e}(L)$, $K \subseteq I \cap J$ such that $\delta|_{K} = \mu|_{K}$. It's obvious that $\mu|_{K} = \delta|_{K}$. Hence $(\mu, J) \equiv (\delta, I)$.

(3) For any $(\delta, I), (\xi, J), (\beta, K) \in \mathscr{D}$ and $(\delta, I) \equiv (\xi, J)$ and $(\xi, J) \equiv (\beta, K)$, there exist $P, Q \in \mathscr{J}_{e}(L)$, $P \subseteq I \cap J$, $Q \subseteq J \cap K$ such that $\delta|_{P} = \xi|_{P}$ and $\xi|_{Q} = \beta|_{Q}$. Since $P, Q \in \mathscr{J}_{e}(L)$, $P \cap Q \neq \emptyset$. Take $T \subseteq P \cap Q$. We also get $T \in \mathscr{J}_{e}(L)$. Moreover, $\delta|_{T} = \xi|_{T} = \beta|_{T}$. Therefore $(\delta, I) \equiv (\beta, K)$.
\end{proof}
\begin{no}\label{no:4.3}
Denote by $Q(L)$ the quotient set $\mathscr{D}/\equiv$. Let $\delta_{I}$ denote the equivalence class of $(\delta, I)$ in $Q(L)$.
\end{no}
\begin{thm}\label{thm:4.4}
Let $L$ be a semiprime Leibniz algebra over $\mathbb{F}$, and let $Q := Q(L)$ be as in Notation \ref{no:4.3}. Define the following maps:
\[\cdot : \mathbb{F} \times Q \rightarrow Q,\;(p, \delta_{I}) \mapsto (p\delta)_{I}\;\rm{where}\; \it{p}\delta : I \rightarrow L,\; y \mapsto \delta(py),\]
\[+ : Q \times Q \rightarrow Q,\;(\delta_{I}, \mu_{J}) \mapsto (\delta + \mu)_{I \cap J}\;\rm{where}\;\delta + \mu : \it{I} \cap J \rightarrow L,\; x \mapsto \delta(x) + \mu(x),\]
\[[\cdot, \cdot] : Q \times Q \rightarrow Q,\; (\delta_{I}, \mu_{J}) \mapsto [\delta, \mu]_{(I \cap J)^{2}}\;\rm{where}\;[\delta, \mu] : \it{(I \cap J)^{2}} \rightarrow L,\; x \mapsto \mu\delta(x) - \delta\mu(x).\]
Then $Q$, with these operations, is a Leibniz algebra containing $L$ as a subalgebra, via the monomorphism:
\[\varphi : L \rightarrow Q,\quad x \mapsto (R_{x})_{L}.\]
\end{thm}
\begin{proof}
(1) For any $\delta_{I}, \mu_{J} \in Q$ and $x, y \in (I \cap J)^{2}$,
\begin{align*}
&[\delta, \mu]([x, y]) = \mu\delta([x, y]) - \delta\mu([x, y]) = \mu([\delta(x), y] + [x, \delta(y)]) - \delta([\mu(x), y] + [x, \mu(y)])\\
&= [\mu\delta(x), y] + [\delta(x), \mu(y)] + [\mu(x), \delta(y)] + [x, \mu\delta(y)] - [\delta\mu(x), y] - [\mu(x), \delta(y)]\\
&- [\delta(x), \mu(y)] - [x, \delta\mu(y)]\\
&= [[\delta, \mu](x), y] + [x, [\delta, \mu](y)],
\end{align*}
which makes sense since $\delta(x)$, $\delta(y)$, $\mu(x)$, $\mu(y)$, $[\delta(x), y]$, $[x, \delta(y)]$, $[\mu(x), y]$, $[x, \mu(y)] \in I \cap J$. Moreover, $\delta([x, y])$, $\mu([x, y]) \in I \cap J$. So $[\delta, \mu] \in \rm{PDer}\it{((I \cap J)^{2}, L)}$, i.e., $[\cdot, \cdot]$ is well-defined.

(2) It's obvious that $(Q, \cdot, +)$ is a vector space over $\mathbb{F}$. For any $\delta_{I}, \mu_{J}, \beta_{K} \in Q$(we can consider the same $T$ for $\delta$, $\mu$ and $\beta$ because if $I, J, K \in \mathscr{J}_{e}(L)$ are such that $\delta : I \rightarrow L$, $\mu : J \rightarrow L$, $\beta : K \rightarrow L$, then $T := I \cap J \cap K \in \mathscr{J}_{e}(L)$ and so $\delta_{I} = \delta_{T}$ and $\mu_{J} = \mu_{T}$ and $\beta_{K} = \beta_{T}$),
\begin{align*}
&[[\delta_{T}, \mu_{T}], \beta_{T}] - [[\delta_{T}, \beta_{T}], \mu_{T}] = [[\delta, \mu]_{T^{2}}, \beta_{T}] - [[\delta, \beta]_{T^{2}}, \mu_{T}]\\
&= [[\delta, \mu], \beta]_{(T^{2} \cap T)^{2}} - [[\delta, \beta], \mu]_{(T^{2} \cap T)^{2}} = ([[\delta, \mu], \beta] - [[\delta, \beta], \mu])_{(T^{2} \cap T)^{2}} = [\delta, [\mu, \beta]]_{(T^{2} \cap T)^{2}}\\
&= [\delta_{T}, [\mu, \beta]_{T^{2}}] = [\delta_{T}, [\mu_{T}, \beta_{T}]].
\end{align*}
Therefore, $Q$ is a Leibniz algebra.

(3) $\varphi$ is well-defined and a homomorphism by the definitions of Leibniz algebra and $[\cdot, \cdot]$. The injectivity of $\varphi$ follows since $(R_{x})_{L} = 0$ for some $x \in L$ means $[L, x] = \{0\}$ and hence $x \in ran(L)$, which is zero by the semiprimeness of $L$ and Proposition \ref{prop:2.6}.
\end{proof}
For any $X \subseteq L$, write $X^{\varphi}$ to denote the image of $X$ inside $Q(L)$ via the monomorphism defined in Theorem \ref{thm:4.4}.
\begin{lem}\label{le:4.5}
For every $\delta_{I} \in Q(L)$, and $(R_{x})_{L} \in I^{\varphi}(x \in I)$, we have $[\delta_{I}, (R_{x})_{L}] = (R_{-\delta(x)})_{L} \in L^{\varphi}$ and $[(R_{x})_{L}, \delta_{I}] = (R_{\delta(x)})_{L} \in L^{\varphi}$.
\end{lem}
\begin{proof}
For any $y \in I$,
\begin{align*}
&[\delta_{I}, (R_{x})_{L}](y) = R_{x}\delta(y) - \delta R_{x}(y) = [\delta(y), x] - \delta([y, x]) = [\delta(y), x] - [\delta(y), x] - [y, \delta(x)]\\
&=  - [y, \delta(x)] = R_{-\delta(x)}(y)
\end{align*}
and so $[\delta_{I}, (R_{x})_{L}] = (R_{-\delta(x)})_{I} = (R_{-\delta(x)})_{L} \in L^{\varphi}$.

Similarly, we have $[(R_{x})_{L}, \delta_{I}] = (R_{\delta(x)})_{L} \in L^{\varphi}$.
\end{proof}
\begin{prop}\label{prop:4.6}
Let $L$ be a semiprime Leibniz algebra. Then $Q(L)$ is semiprime and an algebra of quotients of $L$. Moreover, $Q(L)$ is maximal among the algebras of quotients of $L$, in the sense that if $S$ is an algebra of quotients of $L$, then there exists a monomorphism $\psi : S \rightarrow Q(L)$ which is the identity in $L$. In particular, the map
\[\psi : S \rightarrow Q(L),\quad s \mapsto (R_{s})_{(L : s)}\]
is a monomorphism which is the identity when restricted to $L$.
\end{prop}
\begin{proof}
Take $\delta_{I}, \mu_{I} \in Q(L)$ with $\delta_{I} \neq 0$(we can consider the same $I$ for $\delta$ and $\mu$ because if $J, K \in \mathscr{J}_{e}(L)$ are such that $\delta : J \rightarrow L$, $\mu : K \rightarrow L$, then $I := J \cap K \in \mathscr{J}_{e}(L)$ and so $\delta_{J} = \delta_{I}$ and $\mu_{K} = \mu_{I}$). Choose $a \in I$ such that $\delta(a) \neq 0$. Then $(R_{a})_{L} \in L^{\varphi}$ satisfies $[(R_{a})_{L}, \delta_{I}] = (R_{\delta(a)})_{L} \neq \rm{0}$. Otherwise, $[L, \delta(a)] = \{0\}$, which implies that $\delta(a) \in ran(L) = \{0\}$. And for every $\lambda_{I} \in _{L^{\varphi}}(\mu_{I})$, $[\lambda_{I}, R_{\delta(a)}] = (R_{-\lambda(a)})_{L} \in L^{\varphi}$ and $[R_{a}, \lambda_{I}] = R_{\lambda(a)} \in L^{\varphi}$. This imply that $(R_{a})_{L} \in (L^{\varphi} : \mu_{I})$. Hence, $Q(L)$ is an algebra of quotients of $L^{\varphi}$. The semiprimeness of $Q(L)$ follows by Proposition \ref{prop:3.6} (2).

Now suppose that $S$ is an algebra of quotients of $L$ and consider the map
\[\psi : S \rightarrow Q(L),\quad s \mapsto (R_{s})_{(L : s)}.\]
According to the definition of $(L : s)$, $\psi$ is well-defined. Moreover, $\psi$ is a monomorphism. To prove the injectivity, suppose that $s \in S$ such that $\psi(s) = 0$, that is $[K, s] = \{0\}$ for some ideal $K \in \mathscr{J}_{e}(L)$, $K \subseteq (L : s)$. This implies that $s \in ran_{S}(K)$. Note that $K$ is essential, $ran_{S}(K) = \{\rm{0}\}$ according to Proposition \ref{prop:3.11}, so $s = 0$.
\end{proof}
\begin{defn}
For a semiprime Leibniz algebra $L$, the Leibniz algebra constructed in Theorem \ref{thm:4.4} is called the maximal algebra of quotients of $L$. Denote it by $Q_{m}(L)$.
\end{defn}
The axiomatic characterization of the Martindale ring of quotients given by D. Passman in \cite{P1} has inspired us to give the following description of the maximal algebra of quotients of a semiprime Leibniz algebra.
\begin{prop}\label{prop:4.8}
Let $L$ be a semiprime Leibniz algebra and consider an overalgebra $S$ of $L$. Then there exists a monomorphism between $S$ and $Q_{m}(L)$ which is the identity on $L$, if and only if $S$ satisfies the following properties:
\begin{enumerate}[(1)]
\item For any $s \in S$, there exists $I \in \mathscr{J}_{e}(L)$ such that $[I, s] \subseteq L$.
\item For $s \in S$ and $I \in \mathscr{J}_{e}(L)$, $[I, s] = \{0\}$ implies that $s = 0$.
\item For any $I \in \mathscr{J}_{e}(L)$, $\delta \in \rm{PDer}\it{(I, L)}$, there exists $s \in S$ such that $\delta(x) = [x, s]$ for every $x \in I$.
\end{enumerate}
\end{prop}
\begin{proof}
Define $\psi : S \rightarrow Q_{m}(L)$, $s \mapsto (R_{s})_{I}$ where $I$ is a nonzero essential ideal of $L$ satisfying $[I, s] \subseteq L$; this exists by (1), and so the map is well-defined. Moreover, it's a homomorphism.

The map $\psi$ is injective(if $(R_{s})_{I} = 0$ for some $s \in S$, by (2), $s = 0$) and surjective(consider $\delta_{I} \in Q(L)$, by (3), there exists $s \in S$ such that $\delta(x) = R_{s}(x)$ for each $x \in I$, hence $\delta_{I} = \psi(s)$).

Finally, $\psi$ is the identity on $L$, by identifying $L$ with $L^{\varphi}$, where $\varphi$ is the map defined in Theorem \ref{thm:4.4}.

Conversely, we'll prove that $Q(L)$ satisfies the three conditions.

(1) Consider $q \in Q(L)$. According to Propositions \ref{prop:3.8} (1) and \ref{prop:4.6}, $(L : q) \in \mathscr{J}_{e}(L)$, and by the definition, $[(L : q), q] \subseteq L$.

(2) Take $q \in Q(L)$ and $I \in \mathscr{J}_{e}(L)$ such that $[I, q] = \{0\}$. We have $q \in ran_{Q(L)}(I) = 0$ according to Propositions \ref{prop:3.11} and \ref{prop:4.6}.

(3) Given $I^{\varphi} \in \mathscr{J}_{e}(L^{\varphi})$ and $\bar{\delta} \in \rm{PDer}\it{(I^{\varphi}, L^{\varphi})}$, we have to find $q \in Q(L)$ such that for every $x^{\varphi} \in I^{\varphi}$, $\bar{\delta}(x^{\varphi}) = [x^{\varphi}, q]$. Define
\[\delta : I \rightarrow L,\quad x \mapsto \varphi^{-1}\bar{\delta}((R_{x})_{L}).\]
Then $q := \delta_{I}$ is a partial derivation of $I$ in $L$. Indeed, for any $x, y \in I$,
\begin{align*}
&q([x, y]) = \delta_{I}([x, y]) = \varphi^{-1}\bar{\delta}((R_{[x, y]})_{L}) = \varphi^{-1}\bar{\delta}([(R_{x})_{L}, (R_{y})_{L}])\\
&= \varphi^{-1}([\bar{\delta}((R_{x})_{L}), (R_{y})_{L}] + [(R_{x})_{L}, \bar{\delta}((R_{y})_{L})])\\
&= [\varphi^{-1}\bar{\delta}((R_{x})_{L}), \varphi^{-1}((R_{y})_{L})] + [\varphi^{-1}((R_{x})_{L}), \varphi^{-1}\bar{\delta}((R_{y})_{L})]\\
&= [\delta_{I}(x), y] + [x, \delta_{I}(y)] = [q(x), y] + [x, q(y)].
\end{align*}

Moreover, $[x^{\varphi}, q] = [(R_{x})_{L}, \delta_{I}] = (R_{\delta(x)})_{L} = (\delta(x))^{\varphi} = \bar{\delta}((R_{x})_{L}) = \bar{\delta}(x^{\varphi})$.
\end{proof}
\section{Algebras of quotients of associative algebras generated by right and left multiplication operators of Leibniz algebras}\label{se:5}
In \cite{PS}, authors examined how the notion of algebras of quotients for Lie algebras tied up with the corresponding well-known concept in the associative case. Inspired by the method in \cite{PS}, we mainly study the relationship between Leibniz algebras and the associative algebras generated by right and left multiplication operators of the corresponding Leibniz algebras of quotients. First of all, we will give some definitions and basic notations.

As defined above, $\mathscr{A}(L)$ denotes the associative subalgebra (possibly without identity) of $\rm{End}\it{(L)}$ generated by the elements $R_{x}$ and $L_{y}$ for $x, y$ in $L$.

By an extension of Leibniz algebras $L \subseteq Q$ we will mean that $L$ is a subalgebra of the Leibniz algebra $Q$.

Let $L \subseteq Q$ be an extension of Leibniz algebras and let $\mathscr{A}_{Q}(L)$ be the associative subalgebra of $\mathscr{A}(Q)$ generated by $\{R_{x}, L_{y} : x, y \in L\}$.

Recall that, given an associative algebra $A$ and a subset $X$ of $A$, we define the right annihilator of $X$ in $A$ as
\[ran_{A}(X) = \{a \in A | Xa = 0\},\]
which is always a right ideal of $A$ (and two-sided if $X$ is a right ideal). One similarly defines the left annihilator, which shall be denoted by $lan_{A}(X)$.
\begin{lem}\label{le:5.1}
Let $I$ be an ideal of a Leibniz algebra $L$ with $ran(L) = \{0\}$. Then $\rm{Ann}_{\it{L}}(\it{I}) = \{\rm{0}\}$ if and only if $ran_{\mathscr{A}(L)}(\mathscr{A}_{L}(I)) = \{0\}$.
\end{lem}
\begin{proof}
Suppose that $\rm{Ann}_{\it{L}}(\it{I}) = \{\rm{0}\}$. For any $\mu$ in $ran_{\mathscr{A}(L)}(\mathscr{A}_{L}(I))$, we have $R_{y}\mu = 0$ for any $y \in I$ and $L_{z}\mu = 0$ for any $z \in I$. In particular if $x \in L$, we get
\[0 = R_{y}\mu(x) = [\mu(x), y],\quad0 = L_{z}\mu(x) = [z, \mu(x)],\]
and these imply that $\mu(x) \in \rm{Ann}_{\it{L}}(\it{I}) = \rm{0}$. Hence, $\mu = 0$.

Conversely, suppose that $ran_{\mathscr{A}(L)}(\mathscr{A}_{L}(I)) = \{0\}$. Let $x$ be in $\rm{Ann}_{\it{L}}(\it{I})$. Then for any $y, z$ in $I$ and $u$ in $L$, we have
\[L_{y}R_{x}(u) = [y, [u, x]] = [[y, u], x] + [u, [y, x]] = 0,\]
\[R_{z}R_{x}(u) = [[u, x], z] = [[u, z], x] + [u, [x, z]] = 0,\]
which imply that $R_{x} \in ran_{\mathscr{A}(L)}(\mathscr{A}_{L}(I)) = 0$. Since by assumption $ran(L) = \{0\}$, we obtain $x = 0$.
\end{proof}
For a subset $X$ of an associative algebra $A$, denote by $\langle X \rangle_{A}^{l}$, $\langle X \rangle_{A}^{r}$ and $\langle X \rangle_{A}$ the left, right and two sided ideal of $A$, respectively, generated by $X$.
\begin{lem}\label{le:5.2}
Suppose that $L$ is a subalgebra of $Q$ and $I$ an ideal of $L$. Then
\[\langle \mathscr{A}_{Q}(I) \rangle_{\mathscr{A}_{Q}(L)}^{l} = \langle \mathscr{A}_{Q}(I) \rangle_{\mathscr{A}_{Q}(L)}^{r} = \langle \mathscr{A}_{Q}(I) \rangle_{\mathscr{A}_{Q}(L)}.\]
\end{lem}
\begin{proof}
Notice that given $x \in L$ and $y \in I$, we have
\[R_{x}R_{y} = R_{y}R_{x} + R_{[y, x]},\;R_{x}L_{y} = L_{y}R_{x} + L_{[y, x]},\;L_{x}R_{y} = R_{y}L_{x} + R_{[x, y]},\;L_{x}L_{y} = L_{y}L_{x} + L_{[x, y]}.\]
Then we have $\langle \mathscr{A}_{Q}(I) \rangle_{\mathscr{A}_{Q}(L)}^{l} = \langle \mathscr{A}_{Q}(I) \rangle_{\mathscr{A}_{Q}(L)}^{r}$. Hence we come to the conclusion.
\end{proof}
\begin{lem}\label{le:5.3}
Suppose that $L$ is a subalgebra of $Q$ and $I$ an ideal of $L$. Write $\tilde{I}$ to denote the ideal of $\mathscr{A}_{Q}(L)$ generated by $\mathscr{A}_{Q}(I)$. Then
\begin{enumerate}[(1)]
\item $ran_{\mathscr{A}_{Q}(L)}(\tilde{I}) = ran_{\mathscr{A}_{Q}(L)}(\mathscr{A}_{Q}(I))$.
\item $lan_{\mathscr{A}_{Q}(L)}(\tilde{I}) = lan_{\mathscr{A}_{Q}(L)}(\mathscr{A}_{Q}(I))$.
\end{enumerate}
\end{lem}
\begin{proof}
(1) Since $\mathscr{A}_{Q}(I) \subseteq \tilde{I}$, we have $ran_{\mathscr{A}_{Q}(L)}(\tilde{I}) \subseteq ran_{\mathscr{A}_{Q}(L)}(\mathscr{A}_{Q}(I))$. So it's enough to show that $ran_{\mathscr{A}_{Q}(L)}(\mathscr{A}_{Q}(I)) \subseteq ran_{\mathscr{A}_{Q}(L)}(\tilde{I})$. Let $\lambda \in ran_{\mathscr{A}_{Q}(L)}(\mathscr{A}_{Q}(I))$. By Lemma \ref{le:5.2} we know that, if $\mu \in \tilde{I}$ there exist a natural number $n$, elements $x_{1, i}, \cdots, x_{r_{i}, i} \in L$ and $y_{1, i}, \cdots, y_{s_{i}, i} \in I$ with $0 \leq r_{i} \in \mathbb{N}$ for all $i$ and $\emptyset \neq \{s_{1}, \cdots, s_{n}\} \subseteq \mathbb{N}$, such that
\[\mu = \sum_{i}^{n}\xi_{1, i}\cdots\xi_{r_{i}, i}\eta_{1, i}\cdots\eta_{s_{i}, i}\]
where $\xi_{j, i} = R_{x_{j, i}}$ or $L_{x_{j, i}}$ for $1 \leq j \leq r_{i}$ and $\eta_{k, i} = R_{y_{k, i}}$ or $L_{y_{k, i}}$ for $1 \leq k \leq s_{i}$.

Since $\eta_{s_{i}, i}\lambda = 0$, we see that $\mu\lambda = 0$.

(2) The proof is similar to (1).
\end{proof}
\begin{lem}\label{le:5.4}
Suppose that $L$ is a subalgebra of $Q$ such that $Q$ is a weak algebra of quotients of $L$. Let $I$ be an ideal of $L$. If $\rm{Ann}_{\it{L}}(\it{I}) = \{\rm{0}\}$, then $ran_{\mathscr{A}(Q)}(\mathscr{A}_{Q}(I)) = \{0\}$.
\end{lem}
\begin{proof}
According to Proposition \ref{prop:3.11}, we have $\rm{Ann}_{\it{Q}}(\it{I}) = \{\rm{0}\}$. For any $\mu \in ran_{\mathscr{A}(Q)}(\mathscr{A}_{Q}(I))$, we have $R_{y}\mu = 0$ for any $y \in I$ and $L_{z}\mu = 0$ for any $z \in I$. If $q \in Q$, then we have
\[0 = R_{y}\mu(q) = [\mu(q), y],\quad0 = L_{z}\mu(q) = [z, \mu(q)],\]
which imply that $\mu(q) \in \rm{Ann}_{\it{Q}}(\it{I}) = \rm{0}$, and so $\mu = 0$.
\end{proof}
\begin{lem}\label{le:5.5}
Suppose that $L$ is a subalgebra of $Q$ and let $x_{1}, \cdots, x_{n}$, $y$ in $L$. Then we have, in $\mathscr{A}(Q)$:
\[\xi_{1}\cdots\xi_{n}\eta = \eta\xi_{1}\cdots\xi_{n} + \sum_{i = 1}^{n}\xi_{1}\cdots\xi_{i - 1}\delta_{i}\xi_{i + 1}\cdots\xi_{n}\]
where $\xi_{i} = R_{x_{i}}$ or $L_{x_{i}}$ for $1 \leq i \leq n$, $\eta = R_{y}$ or $L_{y}$ and $\delta_{i}$ is of one of the following forms $R_{[x_{i}, y]}$, $R_{[y, x_{i}]}$, $L_{[x_{i}, y]}$, $L_{[y, x_{i}]}$.

In particular, if $I$ is an ideal of $L$ and $x_{1}, \cdots, x_{n} \in I$, then
\[\xi_{1}\cdots\xi_{n}\eta = \eta\xi_{1}\cdots\xi_{n} + \alpha\]
where $\alpha \in span\{\rho_{1}\cdots\rho_{n}\;|\;\rho_{i} = R_{z_{i}}\;\rm{or}\;\it{L_{z_{i}}}\;\rm{with}\;\it{z_{i}} \in I,\;\rm{1} \leq \it{i} \leq n\}$.
\end{lem}
\begin{proof}
Let's prove the conclusion by induction on $n$. When $n = 1$,
in the case of $\xi_{1} = R_{x_{1}}$ and $\eta = R_{y}$, we have
\[\xi_{1}\eta = R_{x_{1}}R_{y} = R_{y}R_{x_{1}} + R_{[y, x_{1}]} = \eta\xi_{1} + \delta_{1};\]
in the case of $\xi_{1} = R_{x_{1}}$ and $\eta = L_{y}$, we have
\[\xi_{1}\eta = R_{x_{1}}L_{y} = L_{y}R_{x_{1}} + L_{[y, x_{1}]} = \eta\xi_{1} + \delta_{1};\]
in the case of $\xi_{1} = L_{x_{1}}$ and $\eta = R_{y}$, we have
\[\xi_{1}\eta = L_{x_{1}}R_{y} = R_{y}L_{x_{1}} + R_{[x_{1}, y]} = \eta\xi_{1} + \delta_{1};\]
in the case of $\xi_{1} = L_{x_{1}}$ and $\eta = L_{y}$, we have
\[\xi_{1}\eta = L_{x_{1}}L_{y} = L_{y}L_{x_{1}} + L_{[x_{1}, y]} = \eta\xi_{1} + \delta_{1}.\]
Suppose that the conclusion is valid for $n = k$. Then for $n = k + 1$,
\begin{align*}
&\xi_{1}\cdots\xi_{k}\xi_{k + 1}\eta = \xi_{1}\cdots\xi_{k}(\eta\xi_{k + 1} + \delta_{k + 1}) = \xi_{1}\cdots\xi_{k}\eta\xi_{k + 1} + \xi_{1}\cdots\xi_{k}\delta_{k + 1}\\
&= \left(\eta\xi_{1}\cdots\xi_{k} + \sum_{i = 1}^{k}\xi_{1}\cdots\xi_{i - 1}\delta_{i}\xi_{i + 1}\cdots\xi_{k}\right)\xi_{k + 1} + \xi_{1}\cdots\xi_{k}\delta_{k + 1}\\
&= \eta\xi_{1}\cdots\xi_{k + 1} + \sum_{i = 1}^{k + 1}\xi_{1}\cdots\xi_{i - 1}\delta_{i}\xi_{i + 1}\cdots\xi_{k + 1}.
\end{align*}
The proof is completed.
\end{proof}
Let $L$ be a subalgebra of $Q$. Denote by $\mathscr{A}_{0}$ the associative subalgebra of $\mathscr{A}(Q)$ whose elements are those $\mu$ in $\mathscr{A}(Q)$ such that $\mu(L) \subseteq L$. We obviously have the containments:
\[\mathscr{A}_{Q}(L) \subseteq \mathscr{A}_{0} \subseteq \mathscr{A}(Q).\]
\begin{lem}\label{le:5.6}
Suppose that $L$ is a subalgebra of $Q$ and $I$ an ideal of $L$. Let $q_{1}, \cdots, q_{n}$ in $Q$ such that $[q_{i}, I] \subseteq L$ and $[I, q_{i}] \subseteq L$ for every $i = 1, \cdots, n$. Then for $\mu = \xi_{1}\cdots\xi_{n}$ in $\mathscr{A}(Q)$ where $\xi_{i} = R_{q_{i}}$ or $L_{q_{i}}$ for each $1 \leq i \leq n$, we have that $\mu(\tilde{I})^{n} + (\tilde{I})^{n}\mu \subseteq \mathscr{A}_{0}$ (where $(\tilde{I})^{n}$ denotes the $n$-th power of $\tilde{I}$ in the associative algebra $\mathscr{A}_{Q}(L)$).
\end{lem}
\begin{proof}
According to Lemma \ref{le:5.2}, we have
\[(\tilde{I})^{n} = \mathscr{A}_{Q}(L)(\mathscr{A}_{Q}(I))^{n} + (\mathscr{A}_{Q}(I))^{n} = (\mathscr{A}_{Q}(I))^{n}\mathscr{A}_{Q}(L) + (\mathscr{A}_{Q}(I))^{n}.\]
Thus it's enough to prove that, for any $x_{1},\cdots,x_{n} \in I$, $y = \eta_{1}\cdots\eta_{n}$ where $\eta_{i} = R_{x_{i}}$ or $L_{x_{i}}$ for each $1 \leq i \leq n$, both $\mu y$ and $y \mu$ belong to $\mathscr{A}_{0}$.

Let's prove the conclusion by induction on $n$. For $n = 1$, we have $\eta_{1}\xi_{1} = \xi_{1}\eta_{1} + \delta_{1}$ where $\delta_{1}$ is of one of the following forms $R_{[x_{1}, q_{1}]}$, $R_{[q_{1}, x_{1}]}$, $L_{[x_{1}, q_{1}]}$, $L_{[q_{1}, x_{1}]}$. Since $[x_{1}, q_{1}], [q_{1}, x_{1}] \in L$ we see that $\delta_{1} \in \mathscr{A}_{0}$. On the other hand, $\xi_{1}\eta_{1}(L) \subseteq \xi_{1}(I) \subseteq L$, and so $\xi_{1}\eta_{1}, \eta_{1}\xi_{1} \in \mathscr{A}_{0}$.

Assume that the result is true for $n - 1$. Now, by Lemma \ref{le:5.5} we have
\begin{align*}\label{eq:*}
&\eta_{1}\cdots\eta_{n}\xi_{1}\cdots\xi_{n} = (\xi_{1}\eta_{1})\eta_{2}\cdots\eta_{n}\xi_{2}\cdots\xi_{n} + \sum_{i = 1}^{n}\eta_{1}\cdots\eta_{i - 1}\delta_{i}\eta_{i + 1}\cdots\eta_{n}\xi_{2}\cdots\xi_{n}\tag{*}
\end{align*}
where $\delta_{i}$ is of one of the following forms $R_{[x_{i}, q_{1}]}$, $R_{[q_{1}, x_{i}]}$, $L_{[x_{i}, q_{1}]}$, $L_{[q_{1}, x_{i}]}$.

The first summand on the right side belongs to $A_{0}$ since $\xi_{1}\eta_{1} \in \mathscr{A}_{0}$ and $\eta_{2}\cdots\eta_{n}\xi_{2}\cdots\xi_{n} \in \mathscr{A}_{0}$ by induction hypothesis. 

On the other hand, for each of the terms $\eta_{1}\cdots\eta_{i - 1}\delta_{i}\eta_{i + 1}\cdots\eta_{n}\xi_{2}\cdots\xi_{n}$ we have that $x_{i} \in I$ and $[x_{i}, q_{1}], [q_{1}, x_{i}] \in L$. Using Lemma \ref{le:5.5}, we may write this as:
\[\delta_{i}\eta_{1}\cdots\eta_{i - 1}\eta_{i + 1}\cdots\eta_{n}\xi_{2}\cdots\xi_{n} + \alpha\eta_{i + 1}\cdots\eta_{n}\xi_{2}\cdots\xi_{n}\]
where $\alpha \in span\{\rho_{1}\cdots\rho_{i - 1}\;|\;\rho_{j} = R_{z_{j}}\;\rm{or}\;\it{L_{z_{j}}}\;\rm{with}\;\it{z_{j}} \in I,\;\rm{1} \leq \it{j} \leq i - \rm{1}\}$. The induction hypothesis applies to show that this belongs to $\mathscr{A}_{0}$ again. Hence, $y\mu \in \mathscr{A}_{0}$.

If we continue to develop in the expression (\ref{eq:*}), we get, for some $\alpha_{0} \in \mathscr{A}_{0}$,
\[\xi_{1}\xi_{2}\eta_{1}\cdots\eta_{n}\xi_{3}\cdots\xi_{n} + \sum_{i = 1}^{n}\xi_{1}\eta_{1}\cdots\eta_{i - 1}\delta_{i}\eta_{i + 1}\cdots\eta_{n}\xi_{3}\cdots\xi_{n} + \alpha_{0}\]
where $\delta_{i}$ is of one of the following forms $R_{[x_{i}, q_{2}]}$, $R_{[q_{2}, x_{i}]}$, $L_{[x_{i}, q_{2}]}$, $L_{[q_{2}, x_{i}]}$.

Using Lemma \ref{le:5.5} we can write each term of the form
\[\xi_{1}\eta_{1}\cdots\eta_{i - 1}\delta_{i}\eta_{i + 1}\cdots\eta_{n}\xi_{3}\cdots\xi_{n}\]
as:
\[\xi_{1}\delta_{i}\eta_{1}\eta_{2}\cdots\eta_{i - 1}\eta_{i + 1}\cdots\eta_{n}\xi_{3}\cdots\xi_{n} + \xi_{1}\alpha\eta_{i + 1}\cdots\eta_{n}\xi_{3}\cdots\xi_{n},\]
where $\alpha \in span\{\rho_{1}\cdots\rho_{i - 1}\;|\;\rho_{j} = R_{z_{j}}\;\rm{or}\;\it{L_{z_{j}}}\;\rm{with}\;\it{z_{j}} \in I,\;\rm{1} \leq \it{j} \leq i - \rm{1}\}$. Notice that
\[\xi_{1}\delta_{i}\eta_{1} = \xi_{1}\eta_{1}\delta_{i} + \xi_{1}\gamma\]
where $\gamma$ is of one of the following forms $R_{[[x_{i}, q_{2}], x_{1}]}$, $R_{[[q_{2}, x_{i}], x_{1}]}$, $R_{[x_{1}, [x_{i}, q_{2}]]}$, $R_{[x_{1},[q_{2}, x_{i}]]}$, $L_{[[x_{i}, q_{2}], x_{1}]}$, $L_{[[q_{2}, x_{i}], x_{1}]}$, $L_{[x_{1}, [x_{i}, q_{2}]]}$, $L_{[x_{1},[q_{2}, x_{i}]]}$.

Hence, using $[q_{2}, x_{i}], [x_{i}, q_{2}] \in L$ and $x_{i} \in I$, we see that the first summand above belongs to $\mathscr{A}_{0}$. For the second summand, assuming that $\alpha = \rho_{1}\cdots\rho_{i - 1}$ where $\rho_{j} = R_{z_{j}}$ or $L_{z_{j}}$ with $z_{j} \in I$ for $1 \leq j \leq i - 1$, we have $(\xi_{1}\rho_{1})\rho_{2}\cdots\rho_{i - 1}\eta_{i + 1}\cdots\eta_{n}\xi_{3}\cdots\xi_{n}$, which is also an element of $\mathscr{A}_{0}$. Continuing in this way, we find that
\[\eta_{1}\cdots\eta_{n}\xi_{1}\cdots\xi_{n} - \xi_{1}\cdots\xi_{n}\eta_{1}\cdots\eta_{n} \in \mathscr{A}_{0}\]
and by what we have just proved, we see that $\xi_{1}\cdots\xi_{n}\eta_{1}\cdots\eta_{n} \in \mathscr{A}_{0}$.

This completes the proof.
\end{proof}
\begin{cor}\label{cor:5.7}
Suppose that $L$ is a subalgebra of $Q$ and $I$ an ideal of $L$. Let $q_{1}, \cdots, q_{n}$ in $Q$ such that $[q_{i}, I] \subseteq L$ and $[I, q_{i}] \subseteq L$ for every $i = 1, \cdots, n$. Then for $\mu = \xi_{1}\cdots\xi_{n}$ in $\mathscr{A}(Q)$ where $\xi_{i} = R_{q_{i}}$ or $L_{q_{i}}$ for each $1 \leq i \leq n$, we have that $\mu\widetilde{I^{n}} \subseteq A_{0}, \widetilde{I^{n}}\mu \subseteq A_{0}$ (where $I^{n}$ denotes the $n$-th power of $I$ in the Leibniz algebra $L$).
\end{cor}
\begin{proof}
It's straightforward to show that $I^{n}$ is an ideal of $L$ for each $n \geq 1$.

Claim that $\widetilde{I^{n}} \subseteq (\tilde{I})^{n}$. We'll show it by induction on $n$. When $n = 1$, it's obvious. Suppose that $\widetilde{I^{k}} \subseteq (\tilde{I})^{k}$. When $n = k + 1$, for any $x \in I^{k + 1}$, there exist $y \in I^{k}, z \in I$,
\[R_{x} = R_{[y, z]} = R_{z}R_{y} - R_{y}R_{z} \in \widetilde{I^{k}}\tilde{I} \subseteq (\tilde{I})^{k}\tilde{I} = (\tilde{I})^{k + 1},\]
\[L_{x} = L_{[y, z]} = L_{y}L_{z} - L_{z}L_{y} \in \widetilde{I^{k}}\tilde{I} \subseteq (\tilde{I})^{k}\tilde{I} = (\tilde{I})^{k + 1}.\]
According to Lemma \ref{le:5.6}, we come to the conclusion.
\end{proof}
\begin{lem}\label{le:5.8}
Let $L$ be a semiprime Leibniz algebra. If $I$ is an ideal of $L$ with $\rm{Ann}_{\it{L}}(\it{I}) = \{\rm{0}\}$, then $\rm{Ann}_{\it{L}}(\it{I^{s}}) = \{\rm{0}\}$ for any $s \geq 1$. Any finite intersection of ideals with zero annihilator also have zero annihilator.
\end{lem}
\begin{proof}
Let's show it by induction on $s$. When $s = 1$, it's obvious. Suppose that $\rm{Ann}_{\it{L}}(\it{I^{k}}) = \{\rm{0}\}$, i.e., $I^{k}$ is essential. When $s = k + 1$, for any nonzero ideal $J$ of $L$, we have $I^{k} \cap J \neq \{0\}$ since $I^{k}$ is essential. So
\[\{0\} \neq [I^{k} \cap J, I^{k} \cap J] \subseteq [I^{k}, I] \cap J = I^{k + 1} \cap J,\]
which implies that $I^{k + 1}$ is essential and so $\rm{Ann}_{\it{L}}(\it{I^{k + 1}}) = \{\rm{0}\}$.

Similarly, we can show that any finite intersection of ideals with zero annihilator also have zero annihilator by induction.
\end{proof}
\begin{prop}\label{prop:5.9}
Suppose that $L$ is a semiprime subalgebra of $Q$. Then the following conditions are equivalent:
\begin{enumerate}[(1)]
\item $Q$ is an algebra of quotients of $L$;
\item For any $\mu \in \mathscr{A}(Q)\setminus\{0\}$, there exists an ideal $I$ of $L$ with $\rm{Ann}_{\it{L}}(\it{I}) = \{\rm{0}\}$ such that $\mu\tilde{I} \subseteq \mathscr{A}_{0}$ and $\{0\} \neq \tilde{I}\mu \subseteq \mathscr{A}_{0}$. For any $0 \neq q \in Q$, we also have $R_{q}\tilde{I}(L) \neq \{0\}$ or $L_{q}\tilde{I}(L) \neq \{0\}$.
\end{enumerate}
\end{prop}
\begin{proof}
(2) $\Rightarrow$ (1) Let $q \in Q\setminus\{0\}$. Let $\tilde{I}$ be as in (2), so it satisfies $R_{q}\tilde{I}(L) \neq \{0\}$ or $L_{q}\tilde{I}(L) \neq \{0\}$ and both $R_{q}\tilde{I}$ and $L_{q}\tilde{I}$ contain in $\mathscr{A}_{0}$. Set
\[I_{0} := span\{\alpha(x)\;|\;x \in L\; and\; \alpha \in \tilde{I}\}.\]
Then $I_{0}$ is an ideal of $L$ such that $\rm{Ann}_{\it{L}}(\it{I_{\rm{0}}}) = \{\rm{0}\}$. Indeed, if $x, y \in L$ and $\alpha \in \tilde{I}$, we have
\[[y, \alpha(x)] = L_{y}\alpha(x)\;\rm{and}\;\it{L_{y}}\alpha \in \tilde{I},\quad[\alpha(x), y] = R_{y}\alpha(x)\;\rm{and}\;\it{R_{y}}\alpha \in \tilde{I}.\]
According to Proposition \ref{prop:2.8}, we know that for a semiprime Leibniz algebra $L$, $\rm{Ann}_{\it{L}}(\it{I}) = \{\rm{0}\} \Leftrightarrow \it{ran_{L}(I)} = \{\rm{0}\}$ for any ideal $I$ of $L$. If now $[I_{0}, x] = \{0\}$ for $x \in L$, then $R_{x}\tilde{I}(L) = \{0\}$. In particular, for $y, z \in I$, we have $R_{x}R_{y}(z) = \rm{0}$ and so $x \in ran_{L}(I^{2})$, which is zero by Lemma \ref{le:5.8}. So $ran_{L}(I_{0}) = \{0\}$, consequently $\rm{Ann}_{\it{L}}(\it{I_{\rm{0}}}) = \{\rm{0}\}$.

Finally, $[q, I_{0}] = L_{q}\tilde{I}(L) \neq \{0\}$ or $[I_{0}, q] = R_{q}\tilde{I}(L) \neq \{0\}$, and both $[q, I_{0}]$ and $[I_{0}, q]$ contain in $L$, which implies that $Q$ is ideally absorbed into $L$. According to Theorem \ref{thm:3.10}, $Q$ is an algebra of quotients of $L$.

(1) $\Rightarrow$ (2) Let $\mu = \sum_{i \geq 1}\xi_{i, 1}\cdots\xi_{i, r_{i}} \in \mathscr{A}(Q)\setminus\{\rm{0}\}$ where $\xi_{i, j} = R_{q_{i, j}}$ or $L_{q_{i, j}}$ for $1 \leq j \leq r_{i}$ and $q_{i, j} \in Q$. We may of course assume that all $q_{i, j}$ are nonzero elements in $Q$. Set $s = \sum_{i \geq 1}r_{i}$. As $Q$ is an algebra of quotients of $L$, there exists, for every $i$ and $j$, an ideal $J_{i, j}$ of $L$ such that $\rm{Ann}_{\it{L}}(\it{J_{i, j}}) = \{\rm{0}\}$, and $[q_{i, j}, J_{i, j}] \subseteq L, [J_{i, j}, q_{i, j}] \subseteq L$. By Lemma \ref{le:5.8}, the ideal $J = \cap_{i, j}J_{i, j}$ and hence also $I = J^{s}$ have zero annihilator in $L$. Moreover, $[q_{i, j}, J] \subseteq [q_{i, j}, J_{i, j}] \subseteq L$ and $[J, q_{i, j}] \subseteq [J_{i, j}, q_{i, j}] \subseteq L$. According to Corollary \ref{cor:5.7}, we have $\mu\tilde{I} \subseteq \mathscr{A}_{0}$ and $\tilde{I}\mu \subseteq \mathscr{A}_{0}$.

If $\tilde{I}\mu = \{0\}$, then $\mu \in ran_{\mathscr{A}(Q)}(\mathscr{A}_{Q}(I))$ which is zero by Lemma \ref{le:5.4}.

Finally, for any $0 \neq q \in Q$, apply that $Q$ is an algebra of quotients of $I^{2}$ to find $y, z \in I$ such that $[q, [y, z]] \neq 0$ or $u, v \in I$ such that $[[u, v], q] \neq 0$. And $[q, [y, z]] = L_{q}L_{y}(z) \in L_{q}\tilde{I}(L)$(resp. $[[u, v], q] = R_{q}L_{u}(v) \in R_{q}\tilde{I}(L)$).
\end{proof}
Recall that an associative algebra $S$ is a left quotient algebra of a subalgebra $A$ if whenever $p$ and $q \in S$, with $p \neq 0$, there exists $x$ in $A$ such that $xp \neq 0$ and $xq \in A$. An associative algebra $A$ has a left quotient algebra if and only if it has no total right zero divisors different from zero. (Here, an element $x$ in $A$ is a total right zero divisor if $Ax = \{0\}$.)
\begin{lem}\label{le:5.10}\cite{S1}
Let $A$ be a subalgebra of an associative algebra $Q$. Then $Q$ is a left quotient algebra of $L$ if and only if for every nonzero element $q \in Q$ there exists a left ideal $I$ of $A$ with $ran_{A}(I) = \{0\}$ such that $\{0\} \neq Iq \subseteq A$.
\end{lem}
\begin{thm}\label{thm:5.11}
Suppose that $L$ is a semiprime subalgebra of $Q$. Moreover, suppose that $Q$ is an algebra of quotients of $L$. Then $\mathscr{A}(Q)$ is a left quotient algebra of $\mathscr{A}_{0}$.
\end{thm}
\begin{proof}
Let $\mu \in \mathscr{A}(Q)\setminus\{0\}$ and let $I$ be an ideal of $L$ satisfying condition (2) in Proposition \ref{prop:5.9}. Set $J = \mathscr{A}_{0}\tilde{I} + \tilde{I}$, a left ideal of $\mathscr{A}_{0}$ that satisfies $\{0\} \neq J\mu \subseteq \mathscr{A}_{0}$ (because $\{0\} \neq \tilde{I}\mu \subseteq \mathscr{A}_{0}$).

Since also $\rm{Ann}_{\it{L}}(\it{I}) = \{\rm{0}\}$, from Lemma \ref{le:5.4} we obtain that $ran_{\mathscr{A}(Q)}(\mathscr{A}_{Q}(I)) = \{0\}$. This, together with the fact that $\tilde{I}$ contains $\mathscr{A}_{Q}(I)$, implies that $ran_{\mathscr{A}(Q)}(\tilde{I}) = \{0\}$. Since $ran_{\mathscr{A}_{0}}(\tilde{I}) \subseteq ran_{\mathscr{A}(Q)}(\tilde{I})$ we get that $ran_{\mathscr{A}_{0}}(\tilde{I}) = \{0\}$. Hence we get $ran_{\mathscr{A}_{0}}(J) = \{0\}$ since $\tilde{I} \subseteq J$. This concludes the proof according to Lemma \ref{le:5.10}.
\end{proof}
\begin{re}\label{re:5.12}
As established in the proof of the previous result, if $L$ is a subalgebra of $Q$ and if $L$ is semiprime and $I$ is an ideal of $L$ with $\rm{Ann}_{\it{L}}(\it{I}) = \{\rm{0}\}$, then $\mathscr{A}_{0}\tilde{I} + \tilde{I}$ is a left ideal of $\mathscr{A}_{0}$ with zero right annihilator.
\end{re}
\section{Algebras of quotients of Leibniz algebras with dense extensions}\label{se:6}
In this section, we mainly study algebras of quotients of Leibniz algebras via their dense extensions which is introduced by Cabrera in \cite{C1}. We show that dense extension can be lifted from a Leibniz algebra to its essential ideals if the extension is also an algebra of quotients. What's more, we get a conclusion, a converse to Theorem \ref{thm:5.11}, via dense extension of Leibniz algebras. Specifically, for any algebra $L$ (not necessary associative), let $M(L)$ be the associative algebra generated by the identity map together with the operators given by right and left multiplication by elements of $L$. In the case of a Leibniz algebra $L$, note that $M(L)$ is nothing but the unitization of $\mathscr{A}(L)$.

Following \cite{C1}, given an extension of (not necessarily associative) algebra $L \subseteq Q$, the annihilator of $L$ in $M(Q)$ is defined by
\[L^{ann} := \{\mu \in M(Q)\; |\; \mu(x) = 0,\; x \in L\}.\]
If $L^{ann} = \{0\}$, then $L$ is a dense subalgebra of $Q$, and $L \subseteq Q$ is a dense extension of algebras in this sense.
\begin{lem}\label{le:6.1}
Suppose that $L$ is a subalgebra of the Leibniz algebra $Q$ with $\rm{Ann}(\it{Q}) = \{\rm{0}\}$. Then the following conditions are equivalent:
\begin{enumerate}[(1)]
\item $L$ is a dense subalgebra of $Q$.
\item If $\mu(L) = \{0\}$ for some $\mu$ in $\mathscr{A}(Q)$, then $\mu = 0$.
\end{enumerate}
\end{lem}
\begin{proof}
Clearly, (1) implies (2). Conversely, suppose that $\mu \in M(Q)$ satisfies $\mu(L) = \{0\}$. If $\mu(p) \neq 0$ for some $p$ in $Q$, then use $\rm{Ann}(\it{Q}) = \{\rm{0}\}$ to find a nonzero element $q$ in $Q$ satisfying $R_{q}\mu(p) \neq 0$ or $0 \neq s \in Q$ such that $L_{s}\mu(p) \neq 0$. But then $R_{q}\mu(L) = L_{s}\mu(L) = \{0\}$ and since $\mathscr{A}(Q)$ is a two-sided ideal of $M(Q)$, we have that $R_{q}\mu, L_{s}\mu \in \mathscr{A}(Q)$. Hence, condition (2) yields $R_{q}\mu = L_{s}\mu = 0$, a contradiction.
\end{proof}
\begin{lem}\label{le:6.2}\cite{C1}
Suppose that $L$ is a dense subalgebra of $Q$ and $I$ an ideal of $L$. If $\mu$ is an element of $M(Q)$ such that $\mu(I) = \{0\}$, then $\mu(M(Q)(I)) = \{0\}$.
\end{lem}
Following Cabrera and Mohammed in \cite{CM}, we say that a Leibiniz algebra $L$ is multiplicatively semiprime whenever $L$ and its multiplication algebra $M(L)$ are semiprime. Observe that in this situation, and if $L$ is a Leibniz algebra, then $\mathscr{A}(L)$, being an ideal of $M(L)$, will also be a semiprime algebra.
\begin{lem}\label{le:6.3}
Let $L \subseteq Q$ be a dense extension of Leibniz algebras. If $Q$ is multiplicatively semiprime, then $\mathscr{A}_{0}$ is semiprime.
\end{lem}
\begin{proof}
Let $I$ be an ideal of $\mathscr{A}_{0}$ with $I^{2} = \{0\}$. As in the proof of Proposition \ref{prop:5.9}, let $I_{0} = \{\alpha(x)\;\mid\;\alpha \in I, x \in L\}$, which is clearly an ideal of $L$. For any $\mu \in I$, we evidently have $\mu(I_{0}) = \{0\}$ since $I^{2} = \{0\}$. It then follows Lemma \ref{le:6.2} that $\mu(M(Q)(I_{0})) = \{0\}$. This implies that $\mu M(Q)\mu(L) = \{0\}$, and thus $\mu M(Q)\mu = \{0\}$ since $L$ is dense in $Q$. But $M(Q)$ is semiprime by hypothesis, so $\mu = 0$ and since $\mu$ is an arbitrary element in $I$, we get that $I = \{0\}$, that is, $\mathscr{A}_{0}$ is semiprime.
\end{proof}
\begin{lem}\label{le:6.4}\cite{C1}
Suppose that $L$ is a dense subalgebra of $Q$. If $Q$ is multiplicatively semiprime(resp. multiplicatively prime), then $L$ is also multiplicatively semiprime(resp. multiplicatively prime).
\end{lem}
\begin{prop}\label{prop:6.5}
Suppose that $L$ is a dense subalgebra of $Q$ and $Q$ a multiplicatively semiprime algebra of quotients of $L$. Then $I \subseteq Q$ is a dense extension for every essential ideal of $I$ of $L$.
\end{prop}
\begin{proof}
We first observe that $\rm{Ann}(\it{Q}) = \{\rm{0}\}$ since $Q$ is an algebra of quotients of $L$. Hence Lemma \ref{le:6.1} applies. Thus, let $\mu$ be in $\mathscr{A}(Q)$ such that $\mu(I) = \{0\}$, and by way of contradiction assume that $\mu \neq 0$. According to Theorem \ref{thm:5.11}, $\mathscr{A}(Q)$ is a left quotients algebra of $\mathscr{A}_{0}$ and hence there exists $\lambda$ in $\mathscr{A}_{0}$ such that $0 \neq \lambda\mu \in \mathscr{A}_{0}$. Since the extension $L \subseteq Q$ is dense, $\lambda\mu(L) \neq \{0\}$, and since $ran_{L}(I) = \{0\}$, there exists a nonzero element $y$ in $I$ such that $L_{y}\lambda\mu(L) \neq \{0\}$. Using now that $\mathscr{A}(Q)$ has no total right zero divisor according to Lemma \ref{le:5.1}, we get that $\mathscr{A}(Q)L_{y}\lambda\mu \neq \{0\}$, and this, coupled with the semiprimeness of $\mathscr{A}(Q)$, implying that $\mathscr{A}(Q)L_{y}\lambda\mu \mathscr{A}(Q)L_{y}\lambda\mu \neq \{0\}$. A second application of the fact that $L \subseteq Q$ is a dense extension yields $\mathscr{A}(Q)L_{y}\lambda\mu \mathscr{A}(Q)L_{y}\lambda\mu(L) \neq \{0\}$. However, $\mu(I) = \{0\}$ by assumption and thus implies that $\mu M(Q)(I) = \{0\}$. But this is a contradiction, because of the containments
\[\mu \mathscr{A}(Q)L_{y}\lambda\mu(L) \subseteq \mu \mathscr{A}(Q)([I, L]) \subseteq \mu \mathscr{A}(Q)(I) \subseteq \mu M(Q)(I) = \{0\}.\]
This completes the proof.
\end{proof}
\begin{cor}\label{cor:6.6}
Let $L \subseteq Q$ be a dense extension of Leibniz algebras. Suppose that $Q$ is a multiplicatively prime algebra of quotients of $L$. Then $I \subseteq Q$ is a dense extension for any nonzero ideal $I$ of $L$.
\end{cor}
\begin{proof}
According to Proposition \ref{prop:2.9}, we get $ran_{L}(I) = \{0\}$ for each nonzero ideal $I$ of $L$. Since $L$ is prime, we get $I$ is essential where $I$ is a nonzero ideal of $L$. The proof is completed by Proposition \ref{prop:6.5}.
\end{proof}
\begin{cor}\label{cor:6.7}
Suppose that $L$ is a dense subalgebra of $Q$ and $Q$ a multiplicative semiprime algebra of quotients of $L$. Then for every essential ideal $I$ of $L$, $lan_{\mathscr{A}(Q)}(\tilde{I}) = \{0\}$.
\end{cor}
\begin{proof}
Let $\mu \in lan_{\mathscr{A}(Q)}(\tilde{I})$. Then, if $y \in I$, we have $\mu R_{y}(I) = \{0\}$. This implies that $\mu(I^{2}) = \{0\}$. By Proposition \ref{prop:6.5} applied to the essential ideal $I^{2}$ of $L$, the extension $I^{2} \subseteq Q$ is dense, so $\mu = 0$.
\end{proof}
We close by exploring the possible converse to Theorem \ref{thm:5.11} in the presence of dense extensions of Leibniz algebras.
\begin{defn}
Given an extension of associative algebra $A \subseteq S$, we say that $S$ is strong right ideally absorbed into $A$ if for any $p, q \in S\setminus\{0\}$ there is an ideal $I$ of $A$ with $lan_{A}(I) = \{0\}$ and such that $pI \neq \{0\}$ or $qI \neq \{0\}$ and both $pI$ and $qI$ are contained in $A$.
\end{defn}
\begin{prop}\label{prop:6.9}
Let $L \subseteq Q$ be a dense extension of Leibniz algebras with $ran(Q) = lan(Q) = \{0\}$. Suppose that $\mathscr{A}(Q)$ is strong right ideally absorbed into $\mathscr{A}_{0}$. Then $Q$ is an algebra of quotients of $L$.
\end{prop}
\begin{proof}
Let $q \in Q\setminus\{0\}$. Since $ran(Q) = lan(Q) = \{0\}$, we have $R_{q} \neq 0$, $L_{q} \neq 0$. By hypothesis, there is an ideal $I$ of $\mathscr{A}_{0}$ such that $lan_{\mathscr{A}_{0}}(I) = \{0\}$ and $R_{q}I \subseteq \mathscr{A}_{0}$, $L_{q}I \subseteq \mathscr{A}_{0}$ and $R_{q}I \neq \{0\}$ or $L_{q}I \neq \{0\}$. Set $I_{0} = \{\alpha(x)\;|\;\alpha \in I, x \in L\}$, which is an ideal of $L$. Moreover, $\rm{Ann}_{\it{L}}(\it{I_{\rm{0}}}) = \{\rm{0}\}$. Indeed, suppose that an element $x$ in $L$ satisfies $[x, I_{0}] = \{0\}$. By definition, this means that $L_{x}I(L) = \{0\}$, and since the extension is dense we have that $L_{x}I = \{0\}$. Thus $L_{x} \in lan_{\mathscr{A}_{\rm{0}}}(I) = 0$. Note that $lan(L) = \{0\}$, and so $x = 0$, consequently $lan_{L}(I_{0}) = \{0\}$. Thus $\rm{Ann}_{\it{L}}(\it{I_{\rm{0}}}) = \{\rm{0}\}$.

Finally, $[q, I_{0}] = L_{q}\tilde{I}(L) \neq \{0\}$ or $[I_{0}, q] = R_{q}\tilde{I}(L) \neq \{0\}$, and both $[q, I_{0}]$ and $[I_{0}, q]$ contain in $L$, which implies that $Q$ is ideally absorbed into $L$. According to Theorem \ref{thm:3.10}, $Q$ is an algebra of quotients of $L$.
\end{proof}
\begin{cor}\label{cor:6.10}
Suppose that $L \subseteq Q$ is a dense extension of Leibniz algebras. Moreover, suppose that $L$ is semiprime and $ran(Q) = lan(Q) = \{0\}$. If $\mathscr{A}(Q)$ is strong right ideally absorbed into $\mathscr{A}_{0}$, then $\mathscr{A}(Q)$ is a left quotient algebra of $\mathscr{A}_{0}$.
\end{cor}

\end{document}